\documentclass{article}
\usepackage[a4paper, margin=3.5cm]{geometry}
\usepackage{graphicx} 
\usepackage{xcolor}
\usepackage{amsmath,amssymb,tensor,dsfont}
\usepackage[vcentermath]{youngtab}
\usepackage{ytableau}
\usepackage{enumitem}

\usepackage{hyperref}
\hypersetup{
    colorlinks=true,
    linkcolor=blue,
    filecolor=magenta,      
    urlcolor=cyan,
    citecolor=teal,
    pdftitle={Overleaf Example},
    pdfpagemode=FullScreen,
    }

\usepackage{mathrsfs}

\usepackage{amsthm}
\theoremstyle{plain}
\newtheorem{theorem}{Theorem} 
\newtheorem{corollary}{Corollary} 

\newtheorem{lemma}{Lemma} 
\theoremstyle{definition}
\newtheorem{definition}{Definition}
\newtheorem{example}{Example} 
\theoremstyle{remark}
\newtheorem{remark}{Remark} 

\newcommand{\R}{\ensuremath{\mathbb{R}}}

\title{Torsion and semi-degeneracy of\\
    second-order maximally superintegrable systems}

\author{Jeremy Nugent${}^\dagger$ and Andreas Vollmer${}^\S$\\[.3cm]
\begin{minipage}{.85\textwidth}\scshape
    \scriptsize $\dagger$ School of Mathematics and Statistics, The University of New South Wales,
Sydney 2052, Australia.
\hspace*{\fill} e-mail: \texttt{\upshape j.nugent@unsw.edu.au}\\
    \scriptsize $\S$ University of Hamburg, Department of Mathematics, Bundesstra{\ss}e 55, 20146 Hamburg, Germany.\hspace*{\fill} e-mail: \texttt{\upshape andreas.vollmer@uni-hamburg.de} 
\end{minipage} }
\date{\today}

\begin{document}

\maketitle

\begin{abstract}
The isotropic harmonic oscillator and the Kepler-Coulomb system are pivotal models in the Sciences. They are two examples of second-order (maximally) superintegrable (Hamiltonian) systems. These systems are classified in dimension two. A partial classification exists in dimension three.

In this paper, our focus is on second-order superintegrable systems with a $(n+1)$-parameter potential with $n\geq3$.
We find that these systems are underpinned by an information-geometric structure, namely the structure of a statistical manifold with torsion.

We obtain a necessary and sufficient condition for such systems to extend to non-degenerate systems, i.e.\ to admit a maximal family of compatible potentials.
The condition is geometric: we show that a $(n+1)$-parameter potential is the restriction of a non-degenerate potential if and only if a certain trace-free tensor field vanishes.
We interpret this condition as the requirement that a certain affine connection has \emph{vectorial torsion}.
We also show that the condition for a system to be extendable is conformally invariant, allowing us to extend our results to second-order \emph{conformally} superintegrable systems with a $(n+1)$-parameter potential. 
\end{abstract}
\bigskip

\smallskip


\section{Introduction}
Let $(M,g)$ be a Riemannian (smooth) manifold (of dimension $n\geq2$) and $V:M\to\R$ a function on~$M$. Our discussion here is local and we therefore assume that $(M,g)$ is simply connected. We call the function $H:T^*M\to\R$,
\begin{equation}\label{eq:Hamiltonian}
    H(x,p) = g^{ij}(x)p_ip_j+V(x)
\end{equation}
(Einstein's summation convention applies) the \emph{Hamiltonian} on $M$, where $(x,p)$ are canonical Darboux coordinates on $T^*M$.
We denote the canonical Poisson bracket on $M$ by $\{-,-\}$.

Conceptually speaking, a Hamiltonian is superintegrable if there are more symmetries $T^*M\to\R$ that Poisson commute with $H$ than an integrable system.
More precisely, a (maximally) \emph{super\-integrable system} is a Hamiltonian $H=:F^{(0)}$ together with $2n-2$ functions $F^{(\alpha)}:T^*M\to\R$, $1\leq\alpha\leq 2n-2$, such that
\begin{equation}\label{eq:integral}
    \{H,F^{(\alpha)}\} = 0
\end{equation}
for all $1\leq\alpha\leq 2n-2$ and such that
$(F^{(\alpha)})_{0\leq\alpha\leq 2n-2}$
are functionally independent.

\begin{remark}
    Note that the condition~\eqref{eq:integral} automatically holds true for $\alpha=0$.
    We also observe that
    \[
        \left\{H,\sum_{\alpha=0}^{2n-2}c_\alpha F^{(\alpha)}\right\} = 0
    \]
    for all choices $c_\alpha\in\R$, $0\leq \alpha\leq 2n-2$, and we therefore introduce the linear space
    \[
        \mathcal F = \left\{ \sum_{\alpha=0}^{2n-2}c_\alpha F^{(\alpha)} : c_\alpha\in\R \right\}.
    \]
\end{remark}

In the present paper we will be concerned with \emph{second-order superintegrable systems}:
\begin{definition}
    A superintegrable system is called \emph{second order} if the integrals $F^{(\alpha)}$ can be chosen to be quadratic polynomials in the momenta, i.e.\ of the form
    \begin{equation}\label{eq:F}
        F(x,p) = K^{ij}(x)p_ip_j+W(x)
    \end{equation}
    where $K^{ij}$ and $W$ depend on positions only. Without loss of generality, we assume that $K^{ij}=K^{ji}$.
\end{definition}

Note that in a second order system, any $F\in\mathcal F$ is a quadratic polynomial in momenta, and hence $\{H,F\}$ is a cubic polynomial in momenta. The following lemma is easily obtained via a direct computation.
\begin{lemma}
    In a second-order superintegrable system, for a function \eqref{eq:F}:
    \begin{enumerate}
        \item The coefficients $K_{ij}=g_{ia}g_{jb}K^{ab}$ are components of a Killing tensor for $g$, i.e.
        \begin{equation}\label{eq:Killing}
            K_{(ij,k)} = 0
        \end{equation}
        where comma denotes a covariant derivative with respect to the Levi-Civita connection of $g$, and where round brackets denote symmetrisation over the enclosed indices.
        \item The function $W$ satisfies
        \begin{equation}\label{eq:pre-BD}
            W_{,k} = K\indices{_{k}^{a}}V_{,a}
        \end{equation}
    \end{enumerate}
\end{lemma}
For more details on this lemma see \cite{KSV2023}.\\
For brevity we will refer to tensor fields simply by citing the components, i.e.\  $K_{ij}$ satisfying~\eqref{eq:Killing} will be called a Killing tensor, and so forth.
Moreover, by a slight abuse of notation, we identify $K^{ij}$ with $K_{ij}$ etc.\ by virtue of the metric $g_{ij}$.
For ease of notation, we introduce the space
\begin{align}\label{definition:Kspace}
    \mathcal K = \{ K\indices{_{i}^{j}} : F_K=K^{ij}p_ip_j+W_K\in\mathcal F \}\,,
\end{align}
where we suggestively denote by $W_K$ the solution of~\eqref{eq:pre-BD} if the underlying potential or family of potentials $V$ is clear, assuming that the integrability condition of~\eqref{eq:pre-BD} holds, i.e.\ that
\begin{equation}\label{eq:Bertrand-Darboux}
    d(KdV) = 0
\end{equation}
is satisfied for the Killing tensor $K$, interpreted as an endomorphism, cf.\ \cite{bertrand_1857,darboux_1901}. This condition is called the \emph{Bertrand-Darboux condition} and the same as presented in \cite{KKM07,KSV2023}. 
Indeed, Equation~\eqref{eq:Bertrand-Darboux} is obtained by applying the differential to~\eqref{eq:pre-BD}, i.e.\ it is the Ricci identity for~$W$.
Conversely, \eqref{eq:Bertrand-Darboux} has to hold for any Killing tensor arising from a function $F\in\mathcal F$ of the form~\eqref{eq:F}.
The present paper is exclusively concerned with \emph{irreducible} systems.
\begin{definition}
    A second-order superintegrable system is called an \emph{irreducible system} if $\mathcal K$ forms an irreducible set of endomorphisms, i.e.\ such that the $K\indices{^i_j}$ do not have a common eigenvector.
\end{definition}

We also define \emph{non-degenerate systems}, adopting the definition in \cite{KMK2007,KSV2023}, although using a slightly different wording.
We then proceed to define \emph{semi-degenerate systems} analogously to~\cite{ERM17}, where the name was, to our knowledge, first coined, while it appears to first have been discussed already in~\cite{KKM07}, for $n=3$ under the name `three-parameter potential'. As semi-degenerate systems cannot exist in dimension~$2$, we consider dimensions $n\geq3$, see also the introduction of \cite{ERM17}.

\begin{definition}\label{defn:semi-degen}
    Let $(M,g)$ be a Riemannian manifold. Let $\mathcal K$ be a linear subspace in the space of rank-$2$ Killing tensors of $g$ and let $\mathcal V\subset\mathcal C^\infty(M)$ be a linear subspace in the space of functions on $M$, with a functionally linearly independent basis. We assume that $\mathcal K$ forms an irreducible set (of endomorphisms).
    \begin{enumerate}[label=(\roman*)]
        \item
            The tuple $(M,g,\mathcal K,\mathcal V)$ is called a \emph{non-degenerate system}, if for any $V\in\mathcal V$ and $K\in\mathcal K$, and with the Hamiltonian $F^{(0)}:=H=g^{ij}p_ip_j+V$,
            \begin{enumerate}[label=(N\arabic*)]
                \item $K$ and $V$ satisfy~\eqref{eq:Bertrand-Darboux},
                \item there are $2n-2$ elements $K^{(k)}\in\mathcal K$, $1\leq k\leq 2n-2$, such that $(F_{K^{(k)}})_{0\leq k\leq 2n-2}$ are a functionally independent set (\(F_K\) is as defined in \eqref{definition:Kspace}),
                \item $\dim\mathcal V=n+2$.
            \end{enumerate}
        \item\label{item:semi-deg}
            A tuple $(M,g,\mathcal K,\mathcal V)$ is called a \emph{semi-degenerate system}, if for any $V\in\mathcal V$ and $K\in\mathcal K$, and with the Hamiltonian $F^{(0)}:=H=g^{ij}p_ip_j+V$,
            \begin{enumerate}[label=(S\arabic*)]
                \item $K$ and $V$ satisfy~\eqref{eq:Bertrand-Darboux},
                \item there are $2n-2$ elements $K^{(k)}\in\mathcal K$, $1\leq k\leq 2n-2$, such that $(F_{K^{(k)}})_{0\leq k\leq 2n-2}$ are a functionally independent set (\(F_K\) is as defined in \eqref{definition:Kspace}),
                \item\label{item:semi-deg.S3} $\dim\mathcal V=n+1$.
                \item $\mathcal V$ cannot be extended such that $\dim \mathcal V = n+2.$
            \end{enumerate}
    \end{enumerate}
\end{definition}

\noindent A tuple \((M,g,\mathcal K, \mathcal V)\) that satisfies the conditions (S1), (S2), (S3), but not necessarily (S4), is going to be called a \emph{$(n+1)$-parameter system}, following~\cite{KKM07}. Such a system can be either semi-degenerate or extended to a non-degenerate one, depending on whether (S4) is satisfied or not. In Section~\ref{sec:main.result} we give a criterion encoding the data needed to make this distinction.
It has been shown in \cite{KSV2023} that, for an irreducible system, the potential $V$ satisfies
\begin{subequations}\label{eq:Wilczynski}
\begin{align}
    \label{eq:Wilczynski.1}
    V_{,ij} &= T\indices{_{ij}^a}V_{,a} + \frac1n g_{ij} \Delta V
    \\
    \label{eq:Wilczynski.2}
    \frac{n-1}{n} \left( \Delta V\right)_{,k} &= q\indices{_k^a}V_{,a}+\frac1n\,t_k\,\Delta V
\end{align}
where $T\indices{_{ij}^k}$ are components of a tensor field and where $t_k=T\indices{_{ak}^a}$ and
\[ q_{ij} = T\indices{^a_{ij,a}}+T\indices{_i^{ab}}T_{baj}-\mathrm{Ric}_{ij}\,. \]
\end{subequations}
Note that~\eqref{eq:Wilczynski} is a closed (prolongation) system for the potential $V$, i.e.\ an (analytic) solution is specified by $n+2$ constants specifying $V$, $V_{,k}$ and $\Delta V$ in a point on $M$.
Hence non-degeneracy is the maximal case of solutions. In the case of a $(n+1)$-parameter potential, we need to distinguish two possibilities:
\begin{enumerate}[label=(\Alph*)]
    \item There is a space $\mathcal V'$ of potentials, $\mathcal V\subset\mathcal V'$ with $\dim\mathcal V'=n+2$, such that $(M,g,\mathcal K,\mathcal V')$ is a non-degenerate system. We then say that the $(n+1)$-parameter system is \emph{extendable}.
    \item Otherwise we say that it is \emph{non-extendable}, and hence satisfies (S4).
\end{enumerate}
In the present paper, we are concerned with the first of these two cases, which we aim to describe similarly to the treatment of non-degenerate systems in~\cite{KSV2023,KSV2024}. We shall find a necessary and sufficient criterion allowing one to decide if a given $(n+1)$-parameter system is naturally extendable to a non-degenerate one.

\begin{remark}\label{rmk:semi-degeneracy.cond}
    For a $(n+1)$-parameter system there has to be a linear condition of the form
    \begin{equation}\label{eq:semi-degeneracy.condition}
        f^{(0)}\Delta V+\sum_{a=1}^nf^{(a)}V_{,a} = 0
    \end{equation}
    for the Laplacian $\Delta V$ and for the derivatives $V_{,a}$ of the potentials, where the coefficients $f^{(r)}$, $0\leq r\leq n$, depend on the point on $M$. Indeed, a given basis of $\mathcal V$ contains $n+1$ functions, which cannot all be functionally independent.
    We will argue now why, in \eqref{eq:semi-degeneracy.condition}, the coefficient $f^{(0)}$ can, without loss of generality be assumed to be $f^{(0)}=1$.
    To this end, we first note that in a neighborhood where $f^{(0)}\ne0$, we achieve $f^{(0)}=1$ after a division by $f^{(0)}$. If $f^{(0)}=0$ is true only on a zero set, we may omit this set.
    It hence remains to discuss the situation on neighborhoods where $f^{(0)}=0$, i.e.\ we consider a neighborhood where
     \[
        \sum_{a=1}^nf^{(a)}V_{,a} = 0
     \]
    holds. Differentiating once, we find, in that neighborhood,
    \[
        0 = \sum_{a=1}^n(f^{(a)}_{,k}V_{,a}+f^{(a)}V_{,ak})
        = \sum_{a=1}^n\left(
            f^{(a)}_{,k}V_{,a}
            +f^{(a)}T\indices{_{ak}^b}V_{,b}
            +\frac1n\,f^{(a)}g_{ak}\,\Delta V
        \right)\,.
     \]
     Next, introducing suitable functions $\mathfrak{f}^{(b)}_{(k)}$, $0\leq b\leq n$, on $M$,
     \[
        \mathfrak{f}^{(0)}_{(k)}\Delta V+\sum_{b=1}^n\mathfrak{f}^{(b)}_{(k)}V_{,b} = 0\,.
     \]
     Since $\mathfrak{f}^{(0)}_{(k)}=\sum_a\frac1nf^{(a)}g_{ak}$, we can find a linear combination $\mathfrak{f}^{(0)}=\sum_{k=1}^na_k\mathfrak{f}^{(0)}_{(k)}$ such that $\mathfrak{f}^{(0)}\ne0$ in some (possibly smaller) neighborhood, since otherwise all $\mathfrak{f}^{(0)}_{(k)}=0$ and then $f^{(a)}=0$ for all $1\leq a\leq n$, which is inconsistent with the hypothesis.
     Hence, without loss of generality, we assume that each potential $V$ of the $(n+1)$-parameter system satisfies
     \[
        \Delta V=s^kV_{,k}
     \]
     on the domain, where $\hat s=s^{k}\partial_k$ is a vector field on $M$. We denote its associated $1$-form by $s$.
\end{remark}

In \cite{KKM07}, it is shown (for the dimension $n=3$) that a semi-degenerate system will satisfy the Wilczynski equation~\eqref{eq:Wilczynski.nondeg}. As it turns out, this result is independent of dimension. A detailed proof is given in Section 5. Combining the statement in Remark \eqref{rmk:semi-degeneracy.cond} with Equation~\eqref{eq:Wilczynski.1}, we obtain the equation
\begin{align}\label{eq:Semidegen.Wilczynski}
    V_{,ij} = D\indices{_{ij}^m} V_{,m},
\end{align}
where we introduce
\begin{align}\label{eq:n+1.Wilczynski}
    D\indices{_{ij}^m} = T\indices{_{ij}^m} + \frac{1}{n} g_{ij} s^m\,.
\end{align}
We call \eqref{eq:n+1.Wilczynski} the \emph{$(n+1)$-parameter Wilczinski equation}. To avoid misunderstandings, we stress that the tensor $T$ in~\eqref{eq:n+1.Wilczynski} does, a priori, not satisfy the integrability conditions of a non-degenerate structure tensor. Indeed, a non-degenerate structure tensor necessarily decomposes according to
\[
        T_{ijk} = S_{ijk} + \bar t_ig_{jk} + \bar t_j g_{ik} - \frac2n\,g_{ij}\bar t_k
\]
where $\bar t_k$ are components of a $1$-form. However, the trace-free part of $T$ in~\eqref{eq:n+1.Wilczynski} is generally not completely symmetric.
Our main result here is that this, in fact, is the only obstruction on $T$ to satisfy the conditions of a non-degenerate structure tensor.

\begin{example}\label{ex:affine}
    Consider the equations \eqref{eq:Bertrand-Darboux} and~\eqref{eq:Semidegen.Wilczynski}. The simplest possible solution occurs, if \(D\indices{_{ij}^m} = 0\) for $i,j,m\in\{1,2,\dots,n\}$.
    In this case, solving~\eqref{eq:Semidegen.Wilczynski} for local coordinates $x=(x_1,\dots,x_n)$, we obtain the solution \(V = a_0 + \sum_{k=1}^n a_k x_k\), where the $a_k$ are constants. This is an \textit{affine function} $\mathbb R^n\to\mathbb R$. It is straightforwardly verified that~$V$ satisfies \eqref{eq:Bertrand-Darboux} for the Killing tensor fields generated by the symmetric products $dx^i\odot dx^j$, where \(i, j \in \{1, 2, \ldots, n\}\).
    These are the Killing tensors associated to the isotropic harmonic oscillator, which is a non-degenerate system with vanishing structure tensor, i.e.\ with \(T\indices{_{ij}^m}=0\).
    The affine function $V$ therefore defines a \((n+1)\)-parameter system that is extendable to the isotropic harmonic oscillator, which in this context is understood to be defined by the non-degenerate potential
    \begin{equation}\label{eq:non-deg.HO}
        a_0 + b\sum_{k=1}^n x_i^2+ \sum_{k=1}^n a_k x_k
    \end{equation}
    where the $a_k$ and $b$ are constants.
\end{example}
\begin{example}
    Consider~\eqref{eq:non-deg.HO}. In Example~\ref{ex:affine}, we considered the $(n+1)$-dimensional potential obtained from~\eqref{eq:non-deg.HO} via the restriction $b=0$.
    We now consider a different restriction of~\eqref{eq:non-deg.HO}, namely the one obtained by requiring \(a_1=0\), i.e.
    \begin{align*}
        \bar{V} = a_0 + b\sum_{k=1}^n x_i^2+ \sum_{k=2}^n a_k x_k.
    \end{align*}
    This is a $(n+1)$-parameter potential. For the components of the associated structure tensor, we substitute $V=\bar V$ into \eqref{eq:Semidegen.Wilczynski} and obtain that \(D\indices{_{ii}^1} = \frac{1}{x_1}\) where \(i \in \{1, \ldots, n\}\), and otherwise \(D\indices{_{ij}^m} = 0,\).
    Hence, the structure tensor associated to $\bar V$ is non-vanishing.
    Next, consider \eqref{eq:Wilczynski}. Substituting $\bar V$ for $V$, we find that \(T\indices{_{ij}^m} = 0\), as in Example~\ref{ex:affine}. This leads us to our first main result.
\end{example}

\section{Main results.}\label{sec:main.result}

We now formulate the main results of the paper, which are then proved in the remainder of the paper.

\subsection{Extendability of a (n+1)-parameter superintegrable system}

We introduce the so-called \emph{obstruction tensor}, inspired by the \textit{obstruction equation} in \cite{KKM07}, (which holds if and only if the obstruction tensor vanishes in \(n=3\))
\[
    N_{ijk} = \Pi_{(2,1)\circ}\ D_{ijk}
\]
where $\Pi_{(2,1)\circ}$ denotes the projector
\begin{align*}
    \Pi_{(2,1)\circ} M_{ijk} = \frac{1}{3}\left( 2\,M_{(ij)k} -M_{ikj} -M_{jki}\right)
    &+ \frac{2}{3(n-1)} \ g_{ij} m_k
    - \frac{2}{3(n-1)} \ g_{k(i} m_{j)}
\end{align*}
where
$ 
    m_j = M\indices{^i_{ji}} - M\indices{^i_{ij}}\,,
$ 
recalling that round brackets denote symmetrisation in enclosed indices.

\begin{remark}
    Consider the obstruction tensor $N_{ijk}$. We observe that, for dimension $n=2$ of the underlying manifold, $N=0$ holds for dimensional reasons. Indeed, observe that $N$ has hook symmetry and is trace-free. Such tensors are trivially zero in dimension $n=2$, which can easily be confirmed using the hook length formula. We hence re-obtain the following well-known fact about non-extendable $(n+1)$-parameter systems, c.f.~\cite{KKMI05} and~\cite{ERM17}.\footnote{We remark that in the reference~\cite{KKMI05}, the constant potential is not included in the count of the parameters.}\smallskip
    
    \noindent\textbf{\slshape Fact:}
        \emph{Non-extendable $(n+1)$-parameter systems exist only on manifolds of dimension \[ \dim(M)\geq3. \]}\\[-\baselineskip]
    \noindent This fact implies that a $2$-dimensional maximally superintegrable system with a three-parameter potential extends to one with a four-parameter potential. Non-extendable examples with a $(n+1)$-parameter potential are given in~\cite{ERM17}.
\end{remark}

\begin{theorem}\label{thm:main}
    Consider a $(n+1)$-parameter (second-order maximally superintegrable) system on a simply connected manifold of dimension $n\geq3$ with structure tensor $D_{ijk}$.
    It is a restriction of a non-degenerate superintegrable system with structure tensor
    \begin{equation}\label{eq:T.from.D}
        T\indices{_{ij}^k} = D\indices{_{ij}^k} - \frac1n\,g_{ij}\,D\indices{^a_{a}^k}
    \end{equation}
    if and only if the obstruction tensor $N$ vanishes.
\end{theorem}

This theorem is proven below, and shows that (non-)extendability of a \((n+1)\)-parameter system is characterised by the condition $N=0$ ($N \ne 0$).
\medskip

For non-degenerate systems, we investigate the conditions that specify a $(n+1)$-dimensional subspace of potentials. This restriction is encoded in a $1$-form $s$ that satisfies an overdetermined system of partial differential equations. Its general solution depends on $n+1$ integration constants, reflecting the choice of a $(n+1)$-dimensional subspace in the space of potentials associated to the non-degenerate system. Details are given in Section~\ref{sec:restricting.condition}.
\bigskip

The next main result of the paper is the extension of Theorem~\ref{thm:main} for conformally superintegrable systems, i.e.\ the generalization of Theorem~\ref{thm:main} to conformal superintegrability.
A conformally (second-order maximally) superintegrable system is a Hamiltonian~\eqref{eq:Hamiltonian} together with $2n-2$ functions $F^{(\alpha)}:T^*M\to\R$, $1\leq\alpha\leq 2n-2$, such that
\begin{equation}\label{eq:integral.conformal}
    \{H,F^{(\alpha)}\} = 2\varrho^{(\alpha)}\,H
\end{equation}
for all $1\leq\alpha\leq 2n-2$ (where the $\varrho^{(\alpha)}:T^*M\to\R$ are polynomial in momenta) and such that
$(F^{(\alpha)})_{0\leq\alpha\leq 2n-2}$
are functionally independent.
Non-degeneracy, semi-degeneracy and the $(n+1)$-parameter property are then defined analogously to the previously discussed \emph{properly} superintegrable case (precise definitions will be given later). We obtain, similarly, the conformal Wilczynski equations for \(T_{ijk}\) and \(D_{ijk},\) namely,
\begin{align}
    V_{,ij} &= T\indices{_{ij}^m} V_{,m} + \frac{1}{n} g_{ij} \Delta V + \tau_{ij} V\\
    V_{,ij} &= D\indices{_{ij}^m} V_{,m} + \eta_{ij} V
\end{align}
in analogy to \eqref{eq:Wilczynski}, \eqref{eq:Semidegen.Wilczynski}, respectively.
For conformally superintegrable systems, we find the following generalization of Theorem~\ref{thm:main}:

\begin{theorem}\label{thm:main.conformal}
    Consider a $(n+1)$-parameter (second-order maximally conformally superintegrable) system on a simply connected manifold of dimension $n\geq3$ with structure tensors $D_{ijk}$ and $\eta_{ij}$.
    It is the restriction of a non-degenerate conformally superintegrable system with structure tensors 
    \begin{equation}
        T\indices{_{ij}^k} = D\indices{_{ij}^k} - \frac1n\,g_{ij}\,D\indices{^a_{a}^k}
    \end{equation}
    and
    \begin{equation}
        \tau_{ij} = \eta_{ij}-\frac1n\,g_{ij}\,\mathrm{tr}_g(\eta)
    \end{equation}
    if and only if the obstruction tensor $N$ vanishes.
\end{theorem}

The proof of Theorem~\ref{thm:main.conformal} is given in Section~\ref{sec:conformal} and is based on Theorem~\ref{thm:main}.
The conformal invariance of the tensor field $N$ is another crucial ingredient.

\subsection{Interpretation in terms of statistical manifolds with torsion}

We interpret the result of Theorems~\ref{thm:main} and~\ref{thm:main.conformal} geometrically in terms of so-called \emph{statistical manifolds with torsion}.
These are geometric structures with roots in statistics. Applications of statistical manifolds (often with vanishing torsion) include thermodynamics, biology and machine learning, for instance \cite{Shima,Nielsen,MJH2023,Frank2009,AJLS2017}.

\begin{definition}[\cite{Matsuzoe2010}]\label{defn:statistical.manifold}
    A triple $(M,g,\mathcal D)$ consisting of a smooth manifold $M$ with Riemannian metric $g$, and a connection $\mathcal D$ on $\mathfrak X(M)$, is called a \emph{statistical manifold with torsion} if
    \[
        (\mathcal D_Xg)(Y,Z)-(\mathcal D_Yg)(X,Z)+g(\mathcal T(X,Y),Z)=0,
    \]
    where $\mathcal T$ is the torsion of $\mathcal D$.
\end{definition}

We recall that an affine connection is said to \emph{have vectorial torsion} if its torsion $\mathcal T\in\Gamma(\Omega^2(T^*M)\otimes TM)$ is determined by its trace $\mathrm{tr}(\mathcal T)$.

We recall the decomposition
\begin{equation}\label{eq:decompose}
        D_{ijk} = S_{ijk} +N_{ijk}+\frac1n\,g_{ij}s_k+ \bar t_ig_{jk} + \bar t_j g_{ik} - \frac2n\,g_{ij}\bar t_k,
\end{equation}
which allows us to reformulate the necessary and sufficient condition of Theorem~\ref{thm:main.conformal} as follows.

\begin{corollary}\label{cor:torsion}
\begin{enumerate}[label=(\roman*)]
    \item For a $(n+1)$-parameter conformal system with structure tensor $D$ decomposed as in \eqref{eq:decompose}, let
    \[
        \mathcal D=\nabla+D.
    \]
    Its $g$-dual connection $\mathcal D^*$,
    \[
        Z(g(X,Y))=g(\mathcal D_ZX,Y)+g(X,\mathcal D^*_ZY),
    \]
    has vectorial torsion if and and only if the underlying $(n+1)$-parameter system is extendable.
    \item For a $(n+1)$-parameter conformal system with structure tensor $D$ decomposed as in \eqref{eq:decompose}, let
    \begin{equation}\label{eq:good.connection}
        \mathcal D=\nabla+D-\frac1ng\otimes s^\sharp+\frac{n+2}{n}\,g\otimes t^\sharp,
    \end{equation}
    and consider its $g$-dual connection $\mathcal D^*$.
    Then $(M,g,\mathcal D^*)$ defines a statistical manifold with torsion $\mathcal T$, given by
    \[
        g(\mathcal T(X,Y),Z) = g(N(X,Z),Y)-g(N(Y,Z),X).
    \]
    The underlying $(n+1)$-parameter conformal system is extendable if and only if the torsion $\mathcal T$ vanishes.
    %
\end{enumerate}
\end{corollary}
We observe that $\mathcal D$ ``forgets'' about the trace component $s$ of $D$, and it therefore encapsulates the same data as the structure tensor $T$ determined by \eqref{eq:Wilczynski} (analogously for its conformal counterpart).

The corollary allows us to differentiate $(n+1)$-parameter second-order maximally (conformally) superintegrable systems by the torsion of $\mathcal D$:
\begin{itemize}
    \item If the torsion of $\mathcal D$ is zero for a $(n+1)$-parameter system, then the system extends to a non-degenerate system.
    \item If the torsion of $\mathcal D$ is non-zero for a $(n+1)$-parameter system, then the system is a semi-degenerate system.
\end{itemize}
We also observe that for this distinction, the data of the structure tensor $T$ is sufficient. The data encoded in the tensor fields $s^k=g^{ij}D\indices{_{ij}^k}$ and $\mathrm{tr}_g(\eta)$ is not needed.

\subsection{Structure of the paper}

The paper is organised as follows; after introducing some notation, and reviewing some facts about non-degenerate systems in Sections~\ref{sec:young} and~\ref{sec:non-deg}, we discuss the structural equations of a $(n+1)$-parameter potential in Section~\ref{sec:semi-deg}.
In Section~\ref{sec:extendable}, we prove Theorem~\ref{thm:main}.
In Section~\ref{sec:conformal}, we prove Theorem~\ref{thm:main.conformal}.
Section~\ref{sec:torsion.interpretation} interprets these results in terms of the torsion of the underlying statistical manifold structure, cf.\ Corollary~\ref{cor:torsion}.
We conclude the paper in Section~\ref{sec:discussion} with a discussion of the results and an explicit non-extendable example, putting our results into context.

\section{Prerequisites}

\subsection{Young symmetrisers}\label{sec:young}
We introduce Young tableaux, which allow us to denote symmetry projections in a concise and efficient manner, see also~\cite{KSV2023,KSV2024}.
A Young diagram is a collection of (finitely many) boxes that are arranged (from left to right) such that the length of each line is non-increasing from top to bottom. 
Young tableaux are Young diagrams filled by symbols (here: index names). They are useful tools in the representation theory of symmetric and special orthogonal groups. Instead of giving a comprehensive review of Young diagrams and Young tableaux, we shall only introduce the special cases that we are going to use here.

A Young symmetriser, for instance, is a Young tableau consisting of a single row.
We use Young symmetrisers to denote complete symmetrisation in the indices indicated in the tableau, for instance
\[
	\young(ij) \ S_{ijk}
	=S_{ijk}
	+S_{jik}.
\]
Note that the Young symmetriser corresponds to the associated symmetry projector up to a factor of $2$. More precisely, the symmetry projector associated to $\young(ij)$ is $\frac12\,\young(ij)$.
Similarly, a single-column Young tableau denotes complete antisymmetrisation, e.g.
\[
	\young(i,j) \ A_{ijk}
	=A_{ijk}
	-A_{jik}.
\]
A third type of Young tableau is going to be used here, namely the \emph{hook symmetriser} denoted by a Young tableau whose first row extends to the right, while the remaining boxes are confined to the first column, such as in the following example:
\[
	\young(ji,k) \
	T_{ijk}
	=
	\young(ji) \ 
	\young(j,k) \
	T_{ijk}
	=
	\young(ji) \ 
	(T_{ijk}-T_{ikj})
	=T_{ijk}-T_{ikj}
	+T_{jik}-T_{jki}.
\]
Note that such a hook symmetriser is to be read as follows: we first have to apply the Young antisymmetriser given by the first column of the tableau, followed by the application of the Young symmetriser given by the first row of the tableau.

We decorate Young symmetrisers by a subscript $\circ$ in order to refer to the totally trace-free component of the tensor obtained after the respective (not normalised) symmetry projection.
For example, the projection $\Gamma(\mathrm{Sym}_2(T^*M)\otimes T^*M)\to\Gamma(\mathrm{Sym}^3_\circ(T^*M))$ onto the totally symmetric and trace-free part is denoted by $\frac16 \ {\yng(3)}_{\circ}$.

\subsection{Review of non-degenerate systems}\label{sec:non-deg}

Consider a non-degenerate system in dimension $n\geq3$, for which~\eqref{eq:Wilczynski.1} holds. Non-degenerate systems have been investigated thoroughly in \cite{KSV2023}. We review some of their properties, discussed in this reference, confining ourselves to properties relevant for our later discussion of semi-degenerate systems.

First, it is shown in~\cite{KSV2023} that the structure tensor $T\indices{_{ij}^k}$ is determined by the space $\mathcal K$ of Killing tensor fields (respectively, their associated endomorphisms), c.f.~\eqref{eq:Wilczynski}. Moreover, in the reference the integrability conditions of \eqref{eq:Wilczynski} are obtained. These are the Ricci conditions for $V$,
\begin{align*} 
    R\indices{^b_{ijk}} V_{,b} &= V_{,ijk}-V_{,ikj}\,,
    &
    0 &= (\Delta V)_{,kl}-(\Delta V)_{,lk}\,,
\end{align*}
where we may replace second (and third) derivatives of $V$ using \eqref{eq:Wilczynski}.
For a concise notation, we are from now on going to denote antisymmetrisation (in the indicated indices) by the corresponding Young diagram, such that
\begin{align}\label{eq:integrability.non-deg.V.young}
    &\young(j,k) \ V_{,ijk} = R\indices{^m_{ijk}} V_{,m}\,,
    &
    &\young(k,l)  \left( \Delta V \right)_{,kl} = 0\,.
\end{align}
Moreover, we introduce the abbreviations
\begin{align*}
    Q\indices{_{ijk}^m} &:= T\indices{_{ij}^m_{,k}} + T\indices{_{ij}^l} T\indices{_{lk}^m} - R\indices{_{ijk}^m}\,,
    &
    q\indices{_j^m} &:= Q\indices{_{ij}^{im}}\,,
\end{align*}
following \cite{KSV2023}.
The integrability conditions~\eqref{eq:integrability.non-deg.V.young} thus take the form
\begin{align}
        \label{eq:original_integ_conditions.1}
		\underbrace{ \young(j,k) \left( Q\indices{_{ijk}^m}+\frac1{n-1}g_{ij}q\indices{_k^m}\right) }_{(\#1)}V_{,m}
		+\frac1n \ \underbrace{ \young(j,k)
		\left(T_{ijk}+\frac1{n-1}g_{ij}t_k\right) }_{(\#2)}\Delta V&=0\\
	\intertext{and, respectively,}
        \label{eq:original_integ_conditions.2}
		\underbrace{ \young(k,l) \left(q\indices{_k^n_{,l}}+T\indices{_{ml}^n}q\indices{_k^m}+\frac1{n-1}t_kq\indices{_l^n}\right) }_{(\#3)}V_{,n}+\frac1n \ \underbrace{ \young(k,l) \left(t_{k,l}+q_{kl}\right) }_{(\#4)}\Delta V
		&=0.
\end{align}
where each coefficient of $V_{,i}$ and $\Delta V$ has to vanish independently.
For example, it follows from $(\#2)$ that the structure tensor \(T\) can be decomposed as
\begin{equation}\label{eq:decomposition.T.non-deg}
    T_{ijk} = S_{ijk} + \bar t_{i} g_{jk} + \bar t_{j} g_{ik} -\frac2n g_{ij} \bar t_{k}
\end{equation}
where
$$ \bar t_i=\frac{n}{(n-1)(n+2)}\,T\indices{_{ia}^a}\,. $$
Next, $(\#4)$ implies that the 1-form $\bar t_i$ is closed, since $q_{ij}$ is symmetric.
Since we work on a simply connected manifold, we thus have that $\bar t_i$ is exact, $\bar t_i=\bar t_{,i}$ for a function $\bar t$, c.f.~\cite{KSV2023}.
Utilising the decomposition of $T$, it is immediately obtained that
\begin{equation}\label{eq:linear.condition.nondeg}
    \Pi_{(2,1)\circ}\,T_{ijk} = 0\,,\quad \text{i.e.\qquad~$N=0$,}
\end{equation}
where we recall the projection operator introduced earlier.
As we shall explain more thoroughly below, this is the crucial condition that a $(n+1)$-parameter system has to satisfy in order to be naturally extendable to a non-degenerate one. We will be able to make this more precise later.

\subsection{The structure tensor of a \texorpdfstring{$(n+1)$}{(n+1)}-parameter system}\label{sec:semi-deg}

Here we discuss some general properties of $(n+1)$-parameter systems and other preliminaries.
Recall the definition of a $(n+1)$-parameter potential (see below Definition \eqref{defn:semi-degen}). It hence follows, for a $(n+1)$-parameter system, that first derivatives \(V_{,m}\) can be chosen to be linearly independent functions and hence the structure tensor \(D\indices{_{ij}^k}\) of the $(n+1)$-parameter system must be uniquely determined by the space $\mathcal K$ of compatible Killing tensors (c.f.~the analogous reasoning in \cite{KSV2023}), such that
\begin{align}
    V_{,ij} = D\indices{_{ij}^m} V_{,m}. \label{eq:SWEQ}
\end{align}
for a $(1,2)$-tensor field $D\in\Gamma(\mathrm{Sym}_2(T^*M)\otimes TM)$ that we are going to call the $(n+1)$-parameter system's \emph{structure tensor}. We use the same terminology for the associated $(0,3)$-tensor field, which we denote by the same symbol, recalling our earlier convention.
In order to proceed, we decompose $D$ into irreducible parts with respect to the natural action of $\mathrm{GL}_n$.
Since $D$ is symmetric in its first two indices, we have to decompose $\yng(2)\otimes\yng(1)$ into irreducible $\mathrm{GL}_n$-representations:
\[
        \yng(2)\otimes\yng(1)
        \cong
        {\yng(3)}_{\ \circ}\oplus\yng(1)\oplus{\yng(2,1)}_\circ\oplus\yng(1)\,,
\]
where each {\tiny $\yng(1)$} stands for one of the two independent trace components of $D$, and where (as always) $\circ$ denotes the totally trace-free part with respect to the obvious metric.
We identify the component {\tiny ${\yng(2,1)}_\circ$} to be determined by the obstruction tensor $N=\Pi_{(2,1)\circ} D$, and {\tiny ${\yng(3)}_\circ$} by~$S$. 
The two independent trace components are given by
\begin{align*}
    d_k &= D\indices{_{ka}^a}\,,
    &
    s_k &= D\indices{^a_{ak}}\,.
\end{align*}
We hence have the decomposition of $D$ as follows:
\begin{align}\label{eq:D-decomp}
    D_{ijk} = S_{ijk} + N_{ijk} &+ \frac{n}{(n+2)(n-1)} \left( d_i g_{jk} + d_j g_{ik} - \frac{2}{n} g_{ij} d_k \right)
    \nonumber \\
    &- \frac{1}{(n+2)(n-1)} \left( s_i g_{jk} + s_j g_{ik} - (n+1) g_{ij} s_k \right).
\end{align}
Analogously to~\cite{KSV2023}, we now aim to investigate the integrability conditions for \eqref{eq:SWEQ}.
To this end, we compute the covariant derivative of \eqref{eq:SWEQ}, finding
\begin{align}
    V_{,ijk} = \left( D\indices{_{ij}^m_{,k}} + D\indices{_{ij}^a} D\indices{_{ak}^m} \right) V_{,m},
\end{align}
the Ricci identity for this equation yields
\begin{align}
    R\indices{^m_{ijk}} V_{,m} = V_{,ijk} - V_{,ikj} = \young(j,k) \left( D\indices{_{ij}^m_{,k}} + D\indices{_{ij}^a} D\indices{_{ak}^m} \right) V_{,m}.
\end{align}
We hence conclude
\begin{align}
    \young(j,k) \left( D\indices{_{ij}^m_{,k}} + D\indices{_{ij}^a} D\indices{_{ak}^m} - R\indices{_{ijk}^m} \right) V_{,m} = 0,
\end{align}
and, by the assumption that the derivatives \(V_{,m}\) be chosen linearly independently,
\begin{align}\label{eq:V-prolong}
    \young(j,k) \left( D\indices{_{ij}^m_{,k}} + D\indices{_{ij}^a} D\indices{_{ak}^m} - R\indices{_{ijk}^m} \right) = 0,
\end{align}
for all \(m\). This is our first major integrability condition for a $(n+1)$-parameter potential.
Before discussing in more detail the properties of the tensor field $D$, we would like to briefly comment on the relationship of~\eqref{eq:Bertrand-Darboux} and~\eqref{eq:SWEQ}. Note that we obtained~\eqref{eq:SWEQ} as a partial solution of~\eqref{eq:Bertrand-Darboux}. Resubstituting this partial solution into~\eqref{eq:Bertrand-Darboux}, we will later obtain an equation of the form
\begin{equation}\label{eq:shortcut.raw}
    \nabla_kK_{ij} = P\indices{_{ijk}^{ab}}K_{ab},
\end{equation}
where $P\indices{_{ijk}^{ab}}$ is a tensor field, analogous to the case of non-degenerate systems, c.f.~\cite{KSV2023}. Equation~\eqref{eq:shortcut.raw} together with~\eqref{eq:SWEQ} then imply~\eqref{eq:Bertrand-Darboux}.

We return to our main discussion, now addressing some useful properties of the tensor field $D$.
Since we work on a simply connected manifold, the next lemma will allow us to write the trace $d_k$ as differential of a function, $d_k=d_{,k}$.
\begin{lemma}
    The trace $d_k=D\indices{_{ka}^a}$ is closed (and hence locally exact).
\end{lemma}
\begin{proof}
    By taking the trace in \(i,m\) in~\eqref{eq:V-prolong}, we find
\begin{align}\label{eq:dclosedproof1}
    0 &= d_{j,k} + D\indices{_{ij}^a} D\indices{_{ak}^i} - R_{jk} - d_{k,j} - D\indices{_{ik}^a} D\indices{_{aj}^i} + R_{kj}.
\end{align}
Notice that, in the expression \( D\indices{_{ik}^a} D\indices{_{aj}^i}\), both $a$ and $i$ are dummy variables. Switching index names, we thus see that
\begin{align*}
    D\indices{_{ik}^a} D\indices{_{aj}^i}
    &= D\indices{_{ak}^i} D\indices{_{ij}^a}
    = D\indices{_{ij}^a} D\indices{_{ak}^i}.
\end{align*}
Therefore, \eqref{eq:dclosedproof1} can be simplified to
\begin{align*}
    0 &= d_{j,k} + D\indices{_{ij}^a} D\indices{_{ak}^i} - R_{jk} - d_{k,j} - D\indices{_{ik}^a} D\indices{_{aj}^i} + R_{kj}
    \\
    &= d_{j,k} + D\indices{_{ij}^a} D\indices{_{ak}^i} - d_{k,j} - D\indices{_{ij}^a} D\indices{_{ak}^i}
    = d_{j,k} - d_{k,j},
\end{align*}
and thus $d_k$ is closed.
\end{proof}

For the main claim, Theorem~\ref{thm:main}, we will also need the closedness of the $1$-form $t_k=T\indices{_{ka}^a}$ for the tensor $T$ defined in \eqref{eq:T.from.D}.
For this purpose we first prove the following auxiliary lemma.
\begin{lemma}\label{lem:tsclosed}
    For a $(n+1)$-parameter system, the trace $t_k$ is closed, if and only if $s_k$ is closed.
\end{lemma}
\begin{proof}
    Introducing $T_{ijk}$ via formula \eqref{eq:T.from.D}, and then taking the trace in \(i,j\), we obtain
    \begin{align}
        t_j = d_j - \frac{1}{n} s_j.
    \end{align}
    Taking the covariant derivative of this equation,
    \begin{align}
        t_{j,k} = d_{j,k} - \frac{1}{n} s_{j,k},
    \end{align}
    furthermore, skewing this equation in \(j,k\), we find
    \begin{align}
        \young(j,k) \ t_{j,k} = \frac{1}{n} \ \young(j,k) \ s_{k,j}
    \end{align}
    as we have already shown that \(d\) is closed.
    The claim follows.
\end{proof}

Coming back to Theorem~\ref{thm:main}, it hence remains to prove that $s$ is a closed $1$-form in the situation of Theorem~\ref{thm:main}. This is ensured by the following lemma.

\begin{lemma}\label{lem:sformula}
    In a $(n+1)$-parameter system, the $2$-form $ds$ satisfies the formula
    \begin{multline}\label{eq:ds}
        \left( \frac{n-2}{n-1} \right) \young(k,l) \ s_{l,k}
        = \young(k,l) \bigg(
            N\indices{^m_{kl}_{,m}}
            +S\indices{_k^{jm}}(N_{mjl} - N_{mlj})
            \\
            -\frac{n+1}{n+2}\,s^m N_{mkl}
            + \frac{2 n}{(n+2)^2 (n-1)} d^m N_{mkl}
        \bigg)
    \end{multline}
\end{lemma}
\begin{proof}
    In a $(n+1)$-parameter system, we find that \(D\) satisfies \eqref{eq:V-prolong}. Taking the trace in \(i,j\),
    \begin{align*}
        \left( s\indices{^m_{,k}} - D\indices{_{k}^{im}_{,i}}  + s^a D\indices{_{ak}^m} - D\indices{_{k}^{ia}} D\indices{_{ai}^m} + R\indices{_{k}^m} \right) = 0.
    \end{align*}
    Solving for \(s_{m,k},\) then skew symmetrising in \(m,k\), and finally using the decomposition \eqref{eq:D-decomp}, we obtain \eqref{eq:ds}.
\end{proof}

\section{Extendable \texorpdfstring{$(n+1)$}{(n+1)}-parameter systems}\label{sec:extendable}

\subsection{Necessity of the condition in the main theorem}\label{sec:necessity}

In this section, we prove the necessity of $N=0$ in Theorem~\ref{thm:main}, i.e.\ we prove that if a $(n+1)$-parameter system is extendable, then $N$ must vanish.
\begin{lemma}\label{lem:necessity}
    If $T_{ijk}$, defined as in Formula~\eqref{eq:T.from.D}, satisfies the conditions of a non-degenerate structure tensor, then $N=0$.
\end{lemma}
\begin{proof}
    We observe that~\eqref{eq:decomposition.T.non-deg} holds under the hypothesis of the lemma due to \eqref{eq:linear.condition.nondeg}, c.f.~\cite{KSV2023}.    
\end{proof}

\begin{remark}
    Note that we insist here that the structure tensor $T_{ijk}$ of the non-degenerate system is a projection of the structure tensor $D_{ijk}$ such that the potentials of the $(n+1)$-parameter systems are automatically potentials of the non-degenerate system as well.
    The condition $N=0$ is crucial in this regard. Indeed, for a $(n+1)$-parameter system, with structure tensor $D_{ijk}$, consider the tensor
    \[
        T_{ijk} = S_{ijk} + \bar t_ig_{jk} + \bar t_j g_{ik} - \frac2n\,g_{ij}\bar t_k
    \]
    where
    \[
        S_{ijk} := \frac16 \ {\young(ijk)}_\circ D_{ijk}\,,\qquad
        \bar t_i:=\frac{n}{(n-1)(n+2)}\,\left( D\indices{_{ia}^a}
        -\frac1n\,D\indices{^a_{ai}} \right)\,.
    \]
    The tensor field $T_{ijk}$, such defined, would be symmetric and trace-free in its first two indices and would satisfy the linear integrability condition~\eqref{eq:linear.condition.nondeg} of a non-degenerate structure tensor. However, its trace might still fail to be closed, violating the conditions of a non-degenerate structure tensor.
\end{remark}

\begin{example}
    Consider the non-degenerate system given by
    \(g = \mathrm d x^2 +\mathrm d y^2 + \mathrm d z^2, \)
    \begin{align}
        V_{nd} = c_1 (x^2 + y^2 + z^2) + \frac{c_2}{x^2}+\frac{c_3}{y^2}+\frac{c_4}{z^2}+c_5.
    \end{align}
    This example is referred to in the literature as the generic system in three dimensions, labeled [I]. It was first introduced by \cite{SW67} and further studied in \cite{KKMIV06}, for instance.
    It is known that \(V = V_{nd}\) satisfies \(V_{,ij} = T\indices{_{ij}^m} V_{,m} + \frac{1}{n} g_{ij} \Delta V,\) where
    \begin{multline*}
    T = -\frac{2}{x} (\mathrm{d} x)^3 - \frac{2}{y} (\mathrm{d} y)^3 - \frac{2}{z} (\mathrm{d} z)^3 + \frac{1}{x} (\mathrm{d} y \otimes \mathrm{d} y \otimes \mathrm{d} x ) + \frac{1}{x} (\mathrm{d} z \otimes \mathrm{d} z \otimes \mathrm{d} x )
    \\
    + \frac{1}{y} (\mathrm{d} x \otimes \mathrm{d} x \otimes \mathrm{d} y ) + \frac{1}{y} (\mathrm{d} z \otimes \mathrm{d} z \otimes \mathrm{d} y )+ \frac{1}{z} (\mathrm{d} x \otimes \mathrm{d} x \otimes \mathrm{d} z ) + \frac{1}{z} (\mathrm{d} y \otimes \mathrm{d} y \otimes \mathrm{d} z ).
    \end{multline*}
    Consider the restricted potential
    \begin{align}
        V_\text{ext} &= \frac{c_2}{x^2}+\frac{c_3}{y^2}+\frac{c_4}{z^2}+c_5
    \end{align}
    This is an example of an extendable $(n+1)$-parameter potential. Here \(V= V_{ext}\) satisfies \(V\indices{_{,ij}} = D\indices{_{ij}^m} V\indices{_{,m}}\) where
    \begin{align}
        D = -\frac{3}{x} (\mathrm{d}x)^3 -\frac{3}{y} (\mathrm{d}y)^3 -\frac{3}{z} (\mathrm{d}z)^3.
    \end{align}
     One can check that \(D_{ijk} = T_{ijk} + \frac{1}{n} g_{ij} \otimes s_k,\) where 
    \begin{align}
        s_k = -\frac{3}{x} \mathrm{d} x-\frac{3}{y} \mathrm{d} y-\frac{3}{z} \mathrm{d} z.
    \end{align}
    Clearly, this one form is closed, \(s_k = f_{,k}\), where \(f(x,y,z) = -3 \ln(xyz).\)
\end{example}

\subsection{The restricting condition}\label{sec:restricting.condition}

This section is a brief digression from the proof of Theorem~\ref{thm:main}. We have already seen that the vanishing of the obstruction tensor is necessary for the extendability of the $(n+1)$-parameter system.

We now investigate, given a non-degenerate system, how we obtain a $(n+1)$-parameter system from it. In this regard, note that fixing a $(n+1)$-dimensional subspace within $\mathcal V$ requires one linear algebraic condition on the $n+2$ constant coefficients in the basis representation of $V\in\mathcal V$. On the other hand, \eqref{eq:semi-degeneracy.condition} requires us to provide a $1$-form. Here, we show that any such $1$-form is obtained locally from initial data consisting of $n+1$ constants.

Thus, consider a $(n+1)$-parameter system that is extendable, i.e.\ that satisfies $N=0$. We have, c.f.\ Section~\ref{sec:semi-deg}, that
\[
    D_{ijk} = S_{ijk} + \bar t_ig_{jk}+\bar t_jg_{ik}-\frac2n g_{ij}\bar t_k
        +\frac1n s_kg_{ij}
        = T_{ijk}+\frac1n s_kg_{ij}\,.
\]
We are left with the equation, c.f.\ Remark~\ref{rmk:semi-degeneracy.cond},
\begin{equation}\label{eq:Wilczynski.semi.3}
    \Delta V=s^{,k}V_{,k}
\end{equation}
which holds for a $(n+1)$-parameter system.
This equation has to be consistent with~\eqref{eq:Wilczynski}, and in particular with~\eqref{eq:Wilczynski.2}.
Taking another derivative of~\eqref{eq:Wilczynski.semi.3} and equating with~\eqref{eq:Wilczynski.2}, then substituting~\eqref{eq:Wilczynski.semi.3} to eliminate $\Delta V$, we obtain
\begin{equation}\label{eq:dLaplace.V}
    \left( \frac{n}{n-1}q_{ka}+ \frac{1}{n-1}t_ks_{,a}\right)V^{,a}
    = \nabla_k\Delta V = \left( s_{,ak} + s^{,b}D_{bka}\right)V^{,a}\,.
\end{equation}
Note that $q$, $t$ and $T$ only depend on the non-degenerate system underpinning the $(n+1)$-parameter one, while $s$ is specific to the $(n+1)$-parameter system and additional to the non-degenerate system.
Because of condition \ref{item:semi-deg.S3} we conclude
\begin{equation}\label{eq:D²s}
    s_{,ak} = \frac{n}{n-1}\,q_{ka}
    +\left(
        \frac{1}{n-1}t_kg_{ab}
        -T_{bka}
        -\frac1n\,s_{,k} g_{ab}
    \right)s^{,b}\,.
\end{equation}
We verify that the integrability conditions for this equation hold for $(n+1)$-parameter systems with $N=0$.
\begin{lemma}
    The Ricci identities for~\eqref{eq:D²s} are satisfied, if $s_{,k}=D\indices{^a_{ak}}$ with $D$ being the structure tensor of an extendable $(n+1)$-parameter system. 
\end{lemma}
\begin{proof}
    We take one further derivative,
    \begin{multline*}
    s_{ak,b} = \frac{n}{n-1} q_{ka,b} + s\indices{^m_{,b}} \left( \frac{1}{n-1} t_k g_{am} - T_{mka} - \frac{1}{n} s_k g_{am} \right)
    \\
    + s^m \left( \frac{1}{n-1} t_{k,b} g_{am} - T_{mka,b} - \frac{1}{n} s_{k,b} g_{am} \right)
    \end{multline*}
    and substitute this into the Ricci identity,
    \begin{multline*}
        R\indices{^{l}_{akb}}s_{,l}
        =  (s_{,akb}-s_{,abk} )
        \\
        = \young(b,k) \ \left[ \frac{n}{n-1} q_{ka,b} + s\indices{^m_{,b}} \left( \frac{1}{n-1} t_k g_{am} - T_{mka} - \frac{1}{n} s_k g_{am} \right) \right.
        \\
        + \left. s^m \left( \frac{1}{n-1} t_{k,b} g_{am} - T_{mka,b} - \frac{1}{n} s_{k,b} g_{am} \right) \right]\,.
    \end{multline*}
    Second derivatives $s_{,ij}$ can be replaced using~\eqref{eq:D²s}, yielding
    \begin{multline}\label{eq:s.integrability}
        0 = \young(k,b) \ \left(\frac{n}{n-1}  \right) \left( q_{ka,b} + q_{km} T\indices{^m_{ba}} + \frac{1}{n-1} q_{ba} t_k \right)
        \\
        +s^m\  
        \young(k,b)\ \left( T_{mba,k} + T_{mbl} T\indices{^l_{ka}} -R\indices{_{mbka}} + \frac{1}{n}  g_{mb} q_{ka} \right)
        \\
        +\frac{1}{n}\ s^l s_a\ 
        \young(k,b)\ \left( T_{lbk} + \frac{1}{n-1} g_{lb} t_k \right)\,.
    \end{multline}
    We compare this with the conditions of a non-degenerate system, discussed in Section~\ref{sec:non-deg}.
    We observe that the coefficients in~\eqref{eq:s.integrability} indeed vanish independently, due to the conditions $(\#3)$, $(\#1)$ and $(\#2)$, respectively.  
\end{proof}

\begin{remark}
    Given $n+1$ constants and a suitable point $x_0\in M$, we may therefore compute a unique smooth solution $s$ of \eqref{eq:D²s} in a neighborhood around $x_0$. Being provided with a non-degenerate system, we restrict to potentials satisfying \eqref{eq:Wilczynski.semi.3}. It then also satisfies~\eqref{eq:dLaplace.V} by construction. It follows that the space of such solutions to~\eqref{eq:Wilczynski.semi.3} has dimension $n+1$ and we hence obtain a $(n+1)$-parameter system.

    On the other hand, let us specify a (suitable) subspace of dimension $n+1$ within the space of potentials $\mathcal V$ of the non-degenerate system. `Suitable' here means that the constant potentials have to lie within the subspace.
    We may then solve the $n$ independent equations \eqref{eq:Wilczynski.semi.3} for $s_{,k}$ and such $s$ has to satisfy~\eqref{eq:dLaplace.V}.

    We conclude that \eqref{eq:D²s} (together with its Ricci identity) is an adequate equation describing the $1$-forms $s$ that encode the extendable $(n+1)$-parameter system as a subsystem of the underlying non-degenerate one. 
\end{remark}

\subsection{Sufficiency: Extendability of a \texorpdfstring{$(n+1)$}{(n+1)}-parameter system}\label{sec:sufficiency}
It remains to prove that the vanishing of the tensor field $N$ associated to the $(n+1)$-parameter system, i.e.~$N=0$, implies the existence of an additional potential, compatible with any Killing tensor field associated to the $(n+1)$-parameter system.

We will complete this proof in two steps. First, we are going to define a (tentative) non-degenerate structure tensor $T$ from the structure tensor $D$ of the $(n+1)$-parameter system. We then show that this tensor field $T$ satisfies the integrability conditions of a non-degenerate superintegrable potential. Moreover, we verify that the potentials associated to the $(n+1)$-parameter system are solutions of this  system of partial differential equations (PDEs), and that it admits $n+2$ independent smooth solutions. This proves the existence of the extra potential, yielding a non-degenerate family of superintegrable potentials.
In a second step, we verify that this non-degenerate family of potentials is compatible with any of the Killing tensors associated to the $(n+1)$-parameter system. This suffices to prove the claim.

\subsubsection{Extendability of the superintegrable potential}\label{sec:sufficiency.potential}

We begin with the first step in the outlined procedure.
Let, for $n\geq3$,
$$ T\indices{_{ij}^k} = D\indices{_{ij}^{k}} -\frac1n g_{ij} s^{k}
$$
This is a candidate for a non-degenerate structure tensor as it is symmetric in $i$ and $j$ and also trace-free in these indices. However, we need to verify that all integrability conditions of a non-degenerate structure tensor are verified.

Consider the Wilczynski equation~\eqref{eq:Wilczynski.1} of a $(n+1)$-parameter system \eqref{eq:SWEQ}. Computing the trace-free part of it, we find
\begin{equation}\label{eq:Wilczynski.nondeg}
 \nabla^2_{ij}V-\frac1ng_{ij}\Delta V
 = D\indices{_{ij}^k}V_{,k}-\frac1n\,g_{ij}\,s^kV_{,k}
 = T\indices{_{ij}^k}V_{,k} .
\end{equation}
This equation is, by the hypothesis, satisfied for any potential $V$
within the $(n+1)$-dimensional space of compatible potentials.

We take one further differentiation and solve for the derivative of the Laplacian as in Section 4.1 of~\cite{KSV2023}:
\begin{subequations}\label{eq:Wilczynski.semi}
\begin{align}
    \nabla^2_{ij}V &= T\indices{_{ij}^a} V_{,a}
    + \frac{1}{n} g_{ij}\Delta V
    \label{eq:Wilczynski.semi.1}
    \\
    \tfrac{n-1}{n} \nabla_k (\Delta V) &= q\indices{_k^a} V_{,a}
    +\frac{1}{n} t_k \Delta V,
    \label{eq:Wilczynski.semi.2}
\end{align}
\end{subequations}
where 
\begin{align*}
    Q\indices{_{ijk}^m} &:= T\indices{_{ij}^m_{,k}} + T\indices{_{ij}^l} T\indices{_{lk}^m} - R\indices{_{ijk}^m},
    \qquad\text{and}
    &q\indices{_j^m} &:= Q\indices{_{ij}^{im}}.
\end{align*}
The integrability conditions for the above equations are identical to those in \cite{KSV2023}, namely
\begin{align*}
    \young(j,k) \ V_{,ijk} &= R\indices{^m_{ijk}} V_{,m}\,,
    &\young(k,l)  \left( \Delta V \right)_{,kl} &= 0. 
\end{align*}
These equations are of the form:
\begin{align}
		\underbrace{ \young(j,k) \left( Q\indices{_{ijk}^m}+\frac1{n-1}g_{ij}q\indices{_k^m}\right) }_{(*1)} V_{,m}
		+\frac1n \ 
        \underbrace{ \young(j,k)\left(T_{ijk}+\frac1{n-1}g_{ij}t_k\right) }_{(*2)}\Delta V&=0 \label{eq:Qeqn1}\\
	\intertext{and, respectively,}
		\underbrace{ \young(k,l) \left(q\indices{_k^n_{,l}}+T\indices{_{ml}^n}q\indices{_k^m}+\frac1{n-1}t_kq\indices{_l^n}\right) }_{(*3)} V_{,n}
        +\frac1n \ \underbrace{ \young(k,l) \left(t_{k,l}+q_{kl}\right) }_{(*4)} \Delta V
		&=0. \label{eq:Qeqn2}
\end{align}
\textit{Note:} These equations are formally identical to the previously considered Equations \eqref{eq:original_integ_conditions.1} and \eqref{eq:original_integ_conditions.2}. However, we stress that our hypotheses are now different, as our point of departure is a non-degenerate system, i.e.\ these equations already hold by hypothesis, and our objective here is to find the possible choices for~$s$ instead.\\
Since \(N\) is assumed to be zero, $(*2)$ is identically zero, and this requirement is equivalent to~\eqref{eq:linear.condition.nondeg}. Thus \eqref{eq:Qeqn1} reduces to the first summand being zero, and by the assumption~\ref{item:semi-deg.S3} this implies that $(*1)$ vanishes identically.\\
Before approaching \eqref{eq:Qeqn2}, we recall the following lemma, which has been proven in \cite{KSV2023}.
\begin{lemma}\label{lem:TD}
    If \eqref{eq:linear.condition.nondeg} holds, the structure tensor \(T\) can be decomposed into
    \begin{align}
        T_{ijk} = \mathring T_{ijk} + \young(ij) \ \left( \bar{t}_i g_{jk} - \frac{1}{n} g_{ij} \bar{t}_k \right),
    \end{align}
    where \(\mathring T_{ijk}\) is a totally symmetric, trace-free tensor, and \(\bar{t}\) is the rescaled trace as defined earlier.
\end{lemma}

Since \(N\) is identically zero, Lemma \ref{lem:sformula} implies that \(s\) is closed and hence Lemma \ref{lem:tsclosed} implies that \(t\) is closed. If Lemma \ref{lem:TD} holds, it is shown in Corollary 4.6 in \cite{KSV2023} that \(q_{kl}\) is symmetric, hence 
\begin{align*}
    \young(k,l) \ t_{k,l} &= 0,&
    \young(k,l) \ q_{kl} &= 0.
\end{align*}
Thus \eqref{eq:Qeqn2} reduces to its first summand being zero, and by the assumption~\ref{item:semi-deg.S3} this implies that $(*3)$ vanishes identically.
Therefore the integrability conditions \eqref{eq:Qeqn1}, \eqref{eq:Qeqn2} are satisfied for $T$.
We conclude that the system of PDEs \eqref{eq:Wilczynski.semi} can be solved locally for given initial values for $\nabla V$ and $V$ at a point. The solution must therefore depend linearly on $n+2$ parameters, confirming the existence of a non-degenerate family of potentials.
We also see by comparison of the $(n+1)$-parameter Wilczynski equation \eqref{eq:SWEQ} and non-degenerate Wilczynski equation \eqref{eq:Wilczynski.semi.1} that any solution of \eqref{eq:SWEQ} is also a solution of~\eqref{eq:Wilczynski.semi.1}, i.e.~the space of potentials of the initial $(n+1)$-parameter system must lie inside the space of solutions of the PDE system~\eqref{eq:Wilczynski.semi.1}.

\subsubsection{The associated Killing tensors}\label{sec:sufficiency.killing}
In the previous paragraph we have confirmed that a $(n+1)$-parameter family of potentials, for which $N=0$, satisfies the conditions of a non-degenerate potential.
In order to conclude the proof, and show that $T_{ijk}$ satisfies all conditions of a non-degenerate structure tensor, we now also need to ensure that the Killing tensors of the $(n+1)$-parameter system are compatible with the non-degenerate family of potentials.
Compatibility here means that the Bertrand-Darboux condition~\eqref{eq:Bertrand-Darboux} is satisfied for all Killing tensors of the $(n+1)$-parameter system and all solutions $V$ of the prolongation system~\eqref{eq:Wilczynski.semi} obtained in the previous section.
To verify this, recall~\eqref{eq:Wilczynski.nondeg} and substitute it into~\eqref{eq:Bertrand-Darboux}. We obtain
\begin{align}
    0 &= \left[ \nabla_jK_{ia}-\nabla_iK_{ja}
    + K_{ib}D\indices{^{b}_{ja}} - K_{jb}D\indices{^{b}_{ia}} \right]V^{,a}
    \nonumber
    \\
    &= \left[ \nabla_jK_{ia}-\nabla_iK_{ja}
    + K_{ib}T\indices{^b_{ja}} - K_{jb}T\indices{^b_{ia}}
    +\frac1n K_{ib}g\indices{^b_{j}}s_a
    -\frac1n K_{jb}g\indices{^b_{i}}s_a \right]\,V^{,a}
    \nonumber
    \\
    &= \left[ \nabla_jK_{ia}-\nabla_iK_{ja}
    + K_{ib}T\indices{^{b}_{ja}} - K_{jb}T\indices{^{b}_{ia}} \right]V^{,a}
    \label{eq:pre-shortcut}
\end{align}
where $T$ is defined from $D$ via~\eqref{eq:T.from.D}.

\begin{lemma}
    Under the hypotheses of Theorem~\ref{thm:main}, a Killing tensor $K_{ij}$ of the $(n+1)$-parameter system satisfies the prolongation
    \begin{equation}\label{eq:shortcut.N=0}
        \nabla_kK_{ij} = \frac13\,\young(ji,k)\ T\indices{^b_{ji}}K_{bk}\,.
    \end{equation}
\end{lemma}

\begin{proof}
    The proof is analogous to the corresponding proof in~\cite{KSV2024}.
    Since $K_{ij}$ are components of a Killing tensor,
    \[
        \nabla_kK_{ij}+\nabla_jK_{ik}+\nabla_iK_{jk} = 0.
    \]
    Moreover, as due to the hypothesis we are provided with a $(n+1)$-parameter family of potentials, the coefficients of $V^{,a}$ in~\eqref{eq:pre-shortcut} have to vanish independently, i.e.
    \[
        \nabla_jK_{ia}-\nabla_iK_{ja}
        = K_{jb}T\indices{^b_{ia}}
          - K_{ib}T\indices{^b_{ja}}.
    \]
    We hence obtain
    \begin{align*}
        3\nabla_kK_{ij}
        &= \left(
            \nabla_kK_{ij}+\nabla_jK_{ik}+\nabla_iK_{jk}
        \right)
        \\
        &\quad
        +\left( \nabla_jK_{ik}-\nabla_iK_{jk} \right)
        +2\left( \nabla_kK_{ji}-\nabla_jK_{ki} \right)
        \\
        &= \left( K_{jb}T\indices{^b_{ik}}
          - K_{ib}T\indices{^b_{jk}} \right)
        +2\left(
            K_{kb}T\indices{^b_{ji}}
          - K_{jb}T\indices{^b_{ki}}
        \right)
        = \young(ji,k)\ T\indices{^a_{ij}}K_{ak}.
    \end{align*}
\end{proof}
\noindent Note that~\eqref{eq:shortcut.N=0} coincides precisely with the prolongation for the Killing tensors in a non-degenerate system with structure tensor $T_{ijk}$ \cite{KSV2024}.
We have therefore shown, under the hypotheses of Theorem~\ref{thm:main}, that if $K_{ij}$ is a Killing tensor compatible with the $(n+1)$-parameter system with structure tensor $D_{ijk}$, then it is also a Killing tensor compatible with the non-degenerate potential for the non-degenerate structure tensor $T_{ijk}$.
Hence, we have confirmed that there are $2n-1$ functionally independent integrals of the motion (the same as those of the initial $(n+1)$-parameter system) satisfying~\eqref{eq:Bertrand-Darboux} for any solution $V$ of the non-degenerate Wilczynski equation~\eqref{eq:Wilczynski.nondeg}.
Hence Theorem~\ref{thm:main} is proven.

\section{Conformal superintegrability}\label{sec:conformal}

This section is dedicated to the proof of Theorem~\ref{thm:main.conformal}. We will begin with a brief review of conformal superintegrability as needed for our purposes. We note that conformally superintegrable systems generalize superintegrable systems. However, all second-order maximally conformally superintegrable systems are Stäckel equivalent to second-order maximally properly superintegrable systems, cf.\ \cite{Capel_thesis}.

\subsection{Preliminaries}\label{sec:conformal.preliminaries}

Conformally superintegrable systems are classical. We follow the introduction of (second order maximally) conformally superintegrable systems in \cite{KSV2024}.

\begin{definition}
    A \textit{second order conformal superintegrable system} is a Hamiltonian \(H =: F^{(0)}\) together with \(2n-2\) functions \(F^{(\alpha)}: T^*M \rightarrow \mathbb{R}, 1 \leq \alpha \leq 2n-2\) such that
    \begin{align*}
        \{ H, F^{(\alpha)}\} = 2\varrho^{(\alpha)} H
    \end{align*}
    for all \( 1 \leq \alpha \leq 2n-2 \) where \(\varrho^{(\alpha)}\) is a 1-form. Furthermore, \(F^{(\alpha)} = C\indices{^{(\alpha)}_{ij}}(\textbf{q}) p^i p^j + W^{(\alpha)}(\textbf{q})\)
\end{definition}
Similar to the regular superintegrable case, the above condition can be split into two conditions, first cubic in the momenta and the second linear in the momenta:
\begin{align}
    \{C^{(\alpha)},g\} &= 2 \varrho^{(\alpha)} g  \label{CKTE}\\ 
    \{C^{(\alpha)},V\} + \{W^{(\alpha)},g\} &= 2 \varrho^{(\alpha)} V. \label{CBDE}
\end{align}
The first condition being that \(C\) is a conformal Killing tensor, or that there exists a one-form \(\varrho_i\) such that:
\begin{align}
    \young(ijk)\ C_{ij,k} = \young(ijk)\ \varrho_{i}g_{jk}.
\end{align}
The second condition can be written as
\begin{align}\label{condition_conformal:dV}
    d V^{(\alpha)} = C^{(\alpha)}dV + \varrho^{(\alpha)} V.
\end{align}
The integrability condition for the equations above are similar to \eqref{eq:Bertrand-Darboux}, and can be considered the conformal equivalent, in components:
\begin{align}
    \young(i,j)\ \left( C\indices{^m_i}V_{,jm} + C\indices{^m_{i,j}}V_{,m} + \varrho_i V_{,j} + \varrho_{i,j} V \right) = 0
\end{align}
or by taking the derivative of \eqref{condition_conformal:dV}:
\begin{align}\label{eq:conformal_KdV}
    d (C^{(\alpha)}dV) + V d \varrho + dV \wedge \varrho = 0
\end{align}
Following the approach in the regular superintegrable case, we solve this equation for \(V_{,ij}\) and take the prolongation to find the two equations
\begin{align}
    V_{,ij} &= T\indices{_{ij}^a}V_{,a} + \frac1n g_{ij} \Delta V + \tau_{ij} V\\\label{eq:conformal_nd_WEQ}
    \frac{n-1}{n} \left( \Delta V\right)_{,k} &= q\indices{_k^a}V_{,a}+\frac1n\,t_k\,\Delta V + \gamma_{k}V.
\end{align}

\begin{definition}
    Let \((M,g)\) be a Riemannanian manifold. Let \(\mathcal{K}\) be a linear subspace in the space of rank-2 trace free Killing tensors of \(g\) and let \( V \subset \mathcal{C}^{\infty}(M)\) be a linear subspace in the space of functions on \(M\) with a functionally linearly independent basis. We assume that \(\mathcal{K}\) forms an irreducible set (of endomorphisms).
    \begin{enumerate}[label=(\roman*)]
        \item  The tuple \((M,g,\mathcal{K},\mathcal{V})\) is called a \textit{non-degenerate conformal system}, if for any \(V \in \mathcal{V}\) and \( K \in \mathcal{K},\) and with the Hamiltonian \(F^{(0)} := H = g^{ij}p_ip_j + V,\)
        \begin{enumerate}[label=(CN\arabic*)]
            \item \(K\) and \(V\) satisfy \eqref{eq:conformal_KdV}
            \item there are \(2n-2\) trace-free elements \(K^{(k)} \in \mathcal{V}\) such that \((F_{K^{(k)}})_{1 \leq k \leq 2n-2}\) is a functionally independent set (\(F_K\) is as defined in \eqref{definition:Kspace})
            \item dim \(\mathcal{V}=n+2\)
        \end{enumerate}
        \item  The tuple \((M,g,\mathcal{K},\mathcal{V})\) is called a \textit{semi-degenerate conformal system}, if for any \(V \in \mathcal{V}\) and \( K \in \mathcal{K},\) and with the Hamiltonian \(F^{(0)} := H = g^{ij}p_ip_j + V,\)
        \begin{enumerate}[label=(CS\arabic*)]
            \item \(K\) and \(V\) satisfy \eqref{eq:conformal_KdV}
            \item there are \(2n-2\) trace-free elements \(K^{(k)} \in \mathcal{V}\) such that \((F_{K^{(k)}})_{1 \leq k \leq 2n-2}\) is a functionally independent set (\(F_K\) is as defined in \eqref{definition:Kspace})
            \item dim \(\mathcal{V}=n+1\)
            \item \(\mathcal{V}\) cannot be extended such that dim \(\mathcal{V}=n+2\)
        \end{enumerate}
    \end{enumerate}
\end{definition}
Similar to the proper superintegrable case, we find that in the semi-degenerate case there exists an equation that is equivalent to
\begin{equation}\label{eq:conformal.semidegeneracy.condition}
    \Delta V = s^kV_{,k}+\chi\,V\,,
\end{equation}
where $\chi$ is a function and $s^k=g^{ka}s_{a}$ where $s_{a}$ are the components of a $1$-form.
Combining~\eqref{eq:conformal.semidegeneracy.condition} with~\eqref{eq:conformal_nd_WEQ}, we therefore obtain
\begin{align}
    V_{,ij} = D\indices{_{ij}^a} V_{,a} + \eta_{ij}V\,,
\end{align}
where we define
\begin{equation*}
    D\indices{_{ij}^a}:=T\indices{_{ij}^a}+\frac1n\,g_{ij}\,s^{a}\,,\qquad
    \eta_{ij}:=\tau_{ij}+\frac1n\,\chi\,g_{ij}\,.
\end{equation*}

\subsection{Conformal transformations of (n+1)-parameter systems}
\label{sec:conformal.transformations}

We begin by analyzing the transformation behavior of the structure tensor $D$ under conformal rescalings of the Hamiltonian, i.e.\ under the replacement
\[
    H\to\Omega^{-2}H=:\tilde H\,,
\]
which implies $g\to\Omega^2g=:\tilde g$ and $V\to\Omega^{-2}V=:\tilde V$.

A straightforward computation shows
\begin{align*}
    \tilde\nabla^2_{ij}\tilde V
    &= \Omega^{-2}\nabla^2_{ij}V
    +\bigg(
        -3\Upsilon_ig_j^a-3\Upsilon_jg_i^a+g_{ij}\Upsilon^a
    \bigg)\tilde\nabla_a\tilde V
    -2\bigg(
        \nabla_j\Upsilon_i+2\,\Upsilon_i\Upsilon_j
    \bigg)\,\tilde V
    \\
    &= \bigg(
        \underbrace{
        D\indices{_{ij}^a}-3\Upsilon_ig_j^a-3\Upsilon_jg_i^a+g_{ij}\Upsilon^a
        }_{=\tilde D\indices{_{ij}^a}}
    \bigg)\tilde\nabla_a\tilde V
    +\bigg(
        \underbrace{
        \eta_{ij}+2\,D\indices{_{ij}^k}\Upsilon_k
        -2\nabla_j\Upsilon_i-4\,\Upsilon_i\Upsilon_j
        }_{=\tilde \eta_{ij}}
    \bigg)\,\tilde V
\end{align*}
where $\tilde\nabla$ is the Levi-Civita connection of $\tilde g$ and where $\Upsilon_i=\Omega^{-1}\nabla_i\Omega$ as well as $\Upsilon^j=g^{ja}\Upsilon_a$.

We conclude that, in particular, the components of $D$ transform according to
\begin{align*}
    \bar t_i&\to\bar t_i-3\Upsilon_i
    \\
    s_i&\to s_i+(n-6)\Upsilon_i
    \\
    d_i&\to d_i-\frac{3n+2}{n}\Upsilon_i
\end{align*}
We observe that $s$ is conformally invariant if $n=6$, but not otherwise. We therefore define
\[
    b_i:=\frac13(n-6)\bar t_i+s_i\,,
\]
which is conformally invariant. We furthermore define
\[
    \sigma_i:=\frac13\bar t_i\,,
\]
which transforms like $\sigma_i\to\sigma_i-\Upsilon$.

\begin{lemma}\label{lem:conformally.invariant.on.Weyl}
    The data $(S\indices{_{ij}^k},N\indices{_{ij}^k},b_i)$ defines a conformally invariant structure on the Weyl manifold $(M,[(g,\sigma)])$.
\end{lemma}

\subsection{Proof of Theorem~\ref{thm:main.conformal}}
\label{sec:conformal.proof}

We observe that if a semi-degenerate conformal system is extendable, this has to be true for any semi-degenerate conformal system that is conformally equivalent to it, since the size of the space of compatible potentials is not changed under conformal rescalings.

\begin{lemma}\label{lem:conformal.main.1}
    If a semi-degenerate conformally superintegrable system is extendable, then its obstruction tensor vanishes, i.e.\ $N=0$. 
\end{lemma}
\begin{proof}
Due to Theorem~\ref{thm:main}, the condition for the extendability of a proper semi-degenerate system is $N=0$. Now, every conformally superintegrable system can be conformally rescaled (using one of its compatible potentials as Stäckel rescaling factor) into a properly superintegrable system. As discussed in the previous paragraph, the component $N\indices{_{ij}^k}$ in \eqref{eq:decompose} is conformally invariant, after raising one index. Therefore, an extendable semi-degenerate conformal system must have $N=0$.
\end{proof}

\begin{lemma}\label{lem:conformal.main.2}
    If a semi-degenerate conformally superintegrable system has obstruction tensor $N=0$, then it is extendable. 
\end{lemma}
\begin{proof}
    By a suitable conformal rescaling (we use a compatible potential as Stäckel rescaling factor), we obtain a conformally equivalent \textsl{proper} semi-degenerate superintegrable system. By the invariance of $N$, its structure tensor $\tilde D$ has the component $\tilde N$ with respect to the decomposition \eqref{eq:decompose}, where the tilde $\tilde{\ }$ has the obvious meaning. We conclude $\tilde N=0$, and therefore the proper semi-degenerate system is extendable. It follows that also the initial, conformally semi-degenerate system is extendable.
\end{proof}

Lemmas~\ref{lem:conformal.main.1} and~\ref{lem:conformal.main.2} together prove Theorem~\ref{thm:main.conformal}.

\section{Interpretation in terms of torsion}\label{sec:torsion.interpretation}

In this section, we prove Corollary~\ref{cor:torsion}.

\subsection{First claim of the corollary}

We observe that $\mathcal D=\nabla+D$ is torsion-free, since $D$ is symmetric in its lower indices.
Its $g$-dual connection $\mathcal D^*$, in contrast, has torsion $\mathcal T^*$, with
\begin{align*}
    g(\mathcal T^*(X,Y),Z)
    &=g(\mathcal D^*_XY-\mathcal D^*_YX-[X,Y],Z)
    \\
    &= g(D(X,Z),Y)-g(D(Y,Z),X)
    \\
    &=g(N(X,Z),Y)-g(N(Y,Z),X)+g(X,Z)u^\sharp(Y)-g(Y,Z)u^\sharp(X)
\end{align*}
where we let $u_i=\frac1n(s_i-(n+2)t_i)$.
We thus find that the torsion of $\mathcal D^*$ is vectorial if and only if $N=0$.
The first claim of Corollary~\ref{cor:torsion} then follows from Theorem~\ref{thm:main.conformal}.



\subsection{Second claim of the corollary}
Consider the Levi-Civita connection $\nabla$ of $g$, and the connection~\eqref{eq:good.connection}, i.e.\ $\mathcal D=\nabla+\Xi$ where we define
\[
    \Xi\indices{_{ij}^k}
    = D\indices{_{ij}^k}-\frac1ng_{ij}s^k+\frac{n+2}{n}\,g_{ij}t^k
    = S\indices{_{ij}^k}+N\indices{_{ij}^k}+t_ig_j^k+t_jg_i^k+g_{ij}t^k.
\]
We observe that $\mathcal D$ is torsion-free, since $\Xi$ is symmetric in its lower indices.
The $g$-dual connection of $\mathcal D$ is defined by
\[
    Z(g(X,Y))=g(\mathcal D_ZX,Y)+g(X,\mathcal D^*_ZY).    
\]
We observe that $\mathcal D^*$ has torsion
\[
    (\mathcal T^*)_{ij}^k=\Xi\indices{^k_{ij}}-\Xi\indices{^k_{ji}}=N\indices{^k_{ij}}-N\indices{^k_{ji}}.
\]
It remains to verify that $(g,\mathcal D,\mathcal D^*)$ defines the structure of a statistical manifold with torsion on $M$. To this end, we compute
\begin{align*}
    (\mathcal D_Xg)(Y,Z)-(\mathcal D_Yg)(X,Z)
    &= -g(Y,\Xi(X,Z))+g(X,\Xi(Y,Z))
    \\
    &= -g(Y,N(X,Z))+g(X,N(Y,Z))
    = -g(\mathcal T^*(X,Y),Z),
\end{align*}
using the symmetry of $\Xi$, namely $\Xi(X,Y)=\Xi(Y,X)$. Definition~\ref{defn:statistical.manifold} is thus satisfied, i.e.\ $(M,g,\mathcal D^*)$ defines a statistical manifold with torsion $\mathcal T^*$.
We remark that, alternatively, one may argue that $\mathcal D$ is torsion-free and that $\mathcal D^*$ is its dual connection with respect to $g$, drawing upon Proposition 7.1 in \cite{Matsuzoe2010}.
The second claim of Corollary~\ref{cor:torsion} is thus proven.

\section{Discussion and outlook}\label{sec:discussion}

In the present paper we have shown that a second-order (maximally) superintegrable system with $(n+1)$-parameter potential naturally extends to a non-degenerate second-order (maximally) superintegrable system, if and only if $N=0$ holds. We have also shown that this condition extends to the more general conformally superintegrable systems.
Theorems~\ref{thm:main} and~\ref{thm:main.conformal} thus provides an efficient, clean and geometric criterion for the effective investigation of $(n+1)$-parameter systems, determining its non- or semi-degenerate nature.
In Corollary~\ref{cor:torsion} we have re-interpreted this criterion in a geometric way, as a torsion result. Specifically, it can be phrased in terms of a curvature property (``vectorial torsion'') of a naturally defined connection. Moreover, the systems under consideration have a naturally underpinning structure of a statistical manifold with torsion, and the torsion vanishes precisely in the situation that the $(n+1)$-parameter system is extendable.

We mention that this opens the perspective of encoding semi-degenerate systems in terms of a \emph{semi-Weyl structure with torsion}, similar to the formulation of non-degenerate systems via semi-Weyl structures in \cite{Vollmer2025_Weyl}. Indeed, this interpretation is offered by Lemma~\ref{lem:conformally.invariant.on.Weyl}. We denote by $\nabla^{(g,\sigma)}$ the connection defined by $\nabla^{(g,\sigma)}g+2\sigma\otimes g=0$. We may then define the connection
\[
    \mathfrak D_XY:=\nabla^{(g,\sigma)}_XY+S(X,Y)+N(X,Y)+b(X)Y+b(Y)X+g(X,Y)g^{-1}(b)
\]
where $g^ {-1}(b)$ is the vector field associated to the conformally invariant $1$-form $b$. Its $g$-dual $\mathfrak D^\ast$ satisfies the conditions of a semi-Weyl structure with torsion equal to $N$, cf.\ \cite{BN2022}, see also references therein.

As mentioned earlier, all $(n+1)$-parameter systems in dimension $n=2$ are restrictions of non-degenerate systems and thus of the \emph{extendable} type \cite{KKMI05}. Non-extendable examples, on the other hand, exist and are known starting from dimension $n=3$, c.f.~\cite{Evans1990,ERM17}.

\begin{example}\label{ex:non-extendable}
    The generalised Kepler-Coulomb potential with the Hamiltonian
    \begin{align}
        H = p_x^2 + p_y^2 + p_z^2\ + \  \frac{a_1}{x^2} + \frac{a_2}{y^2} + \frac{a_3}{\sqrt{x^2 + y^2 + z^2}}\,,
    \end{align}
    ($x\ne0, y\ne0, x^2+y^2+z^2\ne0$) is an example of a non-extendable $(n+1)$-parameter system, which was first acknowledged in \cite{Evans1990}.
    The structure tensor is straightforwardly obtained as
    \begin{multline}\label{eq:non-extendible}
        D= -\frac{3}{x} (\mathrm{d} x)^3
        - \frac{3}{y} (\mathrm{d} y)^3
        + \frac{\left(x^{2}+y^{2}-2 z^{2}\right)}{z \left(x^{2}+y^{2}+z^{2}\right)} (\mathrm{d} z)^3
        \\
        + \frac{x^2 + 4y^2 + 4z^2}{z(x^2 + y^2 + z^2)} \mathrm{d} x \otimes \mathrm{d} x \otimes \mathrm{d} z
        + \frac{4x^2 + y^2 + 4z^2}{z(x^2 + y^2 + z^2)} \mathrm{d} y \otimes \mathrm{d} y \otimes \mathrm{d} z
        \\
        - \frac{3xy}{z(x^2 + y^2 + z^2)} ( \mathrm{d} x \otimes \mathrm{d} y \otimes \mathrm{d} z + \mathrm{d} y \otimes \mathrm{d} x \otimes \mathrm{d} z)
        \\
        - \frac{3x}{x^2 + y^2 + z^2} (\mathrm{d} x \otimes \mathrm{d} z \otimes \mathrm{d} z + \mathrm{d} z \otimes \mathrm{d} x \otimes \mathrm{d} z)
        \\
        - \frac{3y}{x^2 + y^2 + z^2} (\mathrm{d} y \otimes \mathrm{d} z \otimes \mathrm{d} z + \mathrm{d} z \otimes \mathrm{d} y \otimes \mathrm{d} z)\,,
    \end{multline}
    where $(\mathrm{d}x)^3=\mathrm{d}x\otimes \mathrm{d}x\otimes \mathrm{d}x$ etc.
    By a direct computation using~\eqref{eq:non-extendible}, we obtain
    \begin{multline}
    N = \frac{1}{z(x^2+y^2+z^2)} \Bigl(
        -zy \  \mathrm{d} x \otimes \mathrm{d} x \otimes \mathrm{d} y
        + (y^2-x^2) \  \mathrm{d} x \otimes \mathrm{d} x \otimes \mathrm{d} z
        \\
        + zy \ ( \mathrm{d} x \mathrm{d} y \otimes \mathrm{d} x )
        + zx \ (\mathrm{d} x \mathrm{d} y \otimes \mathrm{d} y )
        - 4xy \ (\mathrm{d} x \mathrm{d} y \otimes \mathrm{d} z)
        + (x^2-y^2) ( \mathrm{d} x \mathrm{d} z \otimes \mathrm{d} x)
        \\
        + 2xy \ (\mathrm{d} x \mathrm{d} z \otimes \mathrm{d} y)
        -zx \ (\mathrm{d} x \mathrm{d} z \otimes \mathrm{d} z)
        - xz (\mathrm{d} y \mathrm{d} y \otimes \mathrm{d} x)
        + (x^2-y^2) ( \mathrm{d} y \mathrm{d} y \otimes \mathrm{d} z)
        \\
        + 2xy (\mathrm{d} y \mathrm{d} z \otimes \mathrm{d} x)
        - (x^2-y^2) ( \mathrm{d} y \mathrm{d} z \otimes \mathrm{d} y)
        -zy ( \mathrm{d} y \mathrm{d} z \otimes \mathrm{d} z)
        + zx (\mathrm{d} z \mathrm{d} z \otimes \mathrm{d} x)
        \\
        + zy (\mathrm{d} z \mathrm{d} z \otimes \mathrm{d} y)
    \Bigr),
    \end{multline}
    where $dxdy=\frac12\left(dx\otimes dy+dy\otimes dx\right)$ etc. This expression is clearly non-vanishing.
    Furthermore, contracting~\eqref{eq:non-extendible} in the first two arguments,
    \begin{equation}\label{eq:s.non-extendible}
        s = -\frac{3}{x} \mathrm{d} x - \frac{3}{y} \mathrm{d} y + \frac{6}{z} \mathrm{d} z\,.
    \end{equation}
    In Lemma~\ref{eq:ds}, we have seen that $N=0$ is sufficient for $s$ to be closed, and hence locally exact. However, the vanishing of $N$ is not necessary for the closedness of $s$. Indeed, the $1$-form $s$ in~\eqref{eq:s.non-extendible} is exact, \(s_k = f_{,k}\), with
    \begin{align*}
        f(x,y,z) &= 
        3\ln\frac{z^2}{xy}\,,
    \end{align*}
    while the condition \(N = 0\) is not satisfied.
\end{example}

\noindent The $(n+1)$-parameter system in Example~\ref{ex:non-extendable} is non-extendable in the second-order sense discussed here, however it does in fact admit an extension to a $(n+2)$-parameter potential, if we allow integrals of the motion of higher order \cite{ERM17}. A further investigation of this phenomenon is left to future research.

With the criterion put forth in Theorem~\ref{thm:main} at hand, non-extendable \emph{second-order} (maximally) superintegrable systems can be identified efficiently among $(n+1)$-parameter systems. Studying these non-extendable systems is going to be a worthwhile enterprise in its own right and, as it is clearly out of the scope of the present paper, it is also left to further research.

\section*{Acknowledgements}

The authors are grateful towards Jonathan Kress and Joshua Capel for discussions.
This research was funded by the German Research Foundation (Deutsche Forschungsgemeinschaft, DFG) through the Research Grant no.~540196982. Andreas Vollmer thanks the University of New South Wales Sydney for hospitality.

\bibliography{reflist}
\bibliographystyle{abbrv}

\end{document}